\documentclass{diacritech_IOS}

\usepackage{amsmath,amssymb}
\usepackage{mathtools}

\usepackage{caption,subcaption} 
\usepackage[all,curve]{xy} 

\usepackage{hyperref} 

\newtheorem{openQuestion}{Open Question}[section]
\DeclarePairedDelimiter{\paren}{(}{)}

\DeclarePairedDelimiter{\bracket}{[}{]}
\DeclarePairedDelimiter{\braces}{\{}{\}}
\DeclarePairedDelimiter{\ceil}{\lceil}{\rceil}
\DeclarePairedDelimiter{\floor}{\lfloor}{\rfloor}

\DeclarePairedDelimiter{\abs}{\vert}{\vert}

\newcommand{\set}[1]{\braces*{#1}}
\DeclarePairedDelimiterX{\setbuild}[2]{\{}{\}}{{#1}\;\delimsize\vert\;{#2}}
\DeclarePairedDelimiterX{\setbuildc}[2]{\{}{\}}{{#1}\vcentcolon{#2}}
\newcommand{\eps}{\varepsilon}
\newcommand{\0}{\emptyset}
\newcommand{\defeq}{\coloneqq}

\newcommand{\what}[1]{\widehat{#1}}
\makeatletter
\newcommand{\nfs}{\@ifnextchar.{}{.\@}}
\makeatother

\newcommand{\ie}{i.e\nfs}
\newcommand{\eg}{e.g\nfs}

\newcommand{\ce}{c.e\nfs}

\DeclareMathSymbol{\R}{\mathbin}{AMSb}{"52} 
\newcommand{\upto}{\mathbin{\!\restriction\!}}

\newcommand{\fineq}{=^{*}}
\newcommand{\leT}{\le_{\rm T}}
\newcommand{\ud}{\overline{\rho}}
\newcommand{\ld}{\underline{\rho}}
\newcommand{\id}{\boldsymbol{\rho}}
\newcommand{\aud}{\overline{\id}}
\newcommand{\ald}{\underline{\id}}

\title{Asymptotic density, immunity, and randomness}
\author{Eric P. Astor\\
\affiliation{Dept\nfs{} of Mathematics, University of Chicago, 5734 S University Ave, Chicago, IL 60637, USA\\
{\tt\href{mailto:epastor@math.uchicago.edu}{epastor@math.uchicago.edu}}\\
{\tt\url{http://math.uchicago.edu/~epastor/}}}}
\date{\today}
\begin{document}
\maketitle

\begin{abstract}
In 2012, inspired by developments in group theory and complexity, Jockusch and Schupp introduced generic computability, capturing the idea that an algorithm might work correctly except for a vanishing fraction of cases. However, we observe that their definition of a negligible set is not computably invariant (and thus not well-defined on the 1-degrees), resulting in some failures of intuition and a break with standard expectations in computability theory.

To strengthen their approach, we introduce a new notion of intrinsic asymptotic density, with rich relations to both randomness and classical computability theory. We then apply these ideas to propose alternative foundations for further development in (intrinsic) generic computability.

Toward these goals, we classify intrinsic density 0 as a new immunity property, specifying its position in the standard hierarchy from immune to cohesive for both general and $\Delta^0_2$ sets, and identify intrinsic density $\frac{1}{2}$ as the stochasticity corresponding to permutation randomness. We also prove that Rice's Theorem extends to all intrinsic variations of generic computability, demonstrating in particular that no such notion considers $\0'$ to be ``computable''.
\end{abstract}

\section{Introduction}
For years, there has been strong interest in the distinction between the idealized world of computation and complexity and that of its real-world applications, particularly in problems or algorithms where we find a separation between the worst-case complexity (or, more broadly, difficulty) and the worst cases actually encountered in practice. The simplex algorithm for linear programming is the classic example; there is a family of examples on which the algorithm takes exponential time \cite{simplexExponential}, yet in practice, every problem actually encountered is solved within polynomial time bounds. Even more extreme examples are known, including several problems in group theory (including some variants of the word problem) that are non-computable in general, but for which a low-complexity algorithm solves all examples encountered in practice. \cite{genericComplexity} In complexity theory, current methods for exploring such structure include the average-case complexity introduced by Gurevich \cite{gurevichAverage} and Levin \cite{levinAverage}, though this is sensitive to one's choice of probability measure, as well as the smoothed analysis of Spielman and Tang \cite{smoothedAnalysis}; however, none of these methods have been adapted to computability theory, and it may well be that none are well-suited to such problems.

Taking a more direct approach, several researchers have begun work on the question of whether an algorithm's problematic behavior might be restricted to a negligible set. This is clearly related to the analysts' notion of ``almost everywhere'', whereby one works modulo sets of measure 0 so as to disregard problematic variations with no practical effect. In a sense, this study is motivated by envy of their methods --- in recent years, we have discovered problems that seem to be ``computable almost everywhere'', and are working to find the right definition for the phrase. In this paper, we take the direct approach, studying a new definition of negligibility as applied to the non-negative integers; we will spend most of our time fitting this idea into its proper computability-theoretic context, and then lay the foundations for further investigation into our motivating problem.

The essential difficulty in defining ``computable almost everywhere'' is that there is no uniform probability measure on the integers, and thus no natural notion of a null set. Instead, if we want a uniform measurement of the size of a subset of $\omega$, we are forced to abandon countable additivity and fall back to pseudo-measures. One of the most practical is asymptotic density.

\begin{defn}
Let $S\subseteq\omega$, where $\omega=\set{0,1,2,\ldots}$ is the set of natural numbers. For every $n\ge 0$, we denote $S\cap[0,n)$ by $S\upto n$.

We define the $n$-th \emph{partial density} of $S$ as
\begin{equation*}
\rho_n(S)\defeq\frac{\abs{S\upto n}}{n}.
\end{equation*}
The \emph{lower density} $\ld(S)$ of $S$ is
\begin{equation*}
\ld(S)\defeq\liminf_{n\to\infty}{\rho_n(S)}=\liminf_{n\to\infty}{\frac{\abs{S\upto n}}{n}},
\end{equation*}
and the \emph{upper density} $\ud(S)$ of $S$ is
\begin{equation*}
\ud(S)\defeq\limsup_{n\to\infty}{\rho_n(S)}=\limsup_{n\to\infty}{\frac{\abs{S\upto n}}{n}}.
\end{equation*}

If the limit of the partial densities exists (\ie{}, $\ld(S)=\ud(S)$), then we say that $S$ has \emph{(asymptotic) density}
\begin{equation*}
\rho(S)\defeq\lim_{n\to\infty}{\rho_n(S)}=\lim_{n\to\infty}{\frac{\abs{S\upto n}}{n}}.
\end{equation*}
\end{defn}

Of course, $0\le\ld(S)\le\ud(S)\le 1$ for all $S\subseteq\omega$. In an unfortunate collision of terms, at least for computability-theoretic work, a set is said to be \emph{generic} if it has density 1 (equivalently, $\ld(S)=1$). The name is motivated by the fact that given a generic set $S$, the probability that a random integer selected from $\left[0,n\right)$ will lie in $S$ approaches 1 as $n$ increases; thus, in some sense, such a set contains all generic integers. The complement of this notion is more useful for our purposes:

\begin{defn}
A set $S\subset\omega$ is said to be \emph{negligible} if it has density 0 (equivalently, if $\ud(S)=0$).
\end{defn}

In 2003, Kapovich, Myasnikov, Schupp, and Shpilrain introduced generic-case complexity \cite{genericComplexity}, considering problems modulo sets of density 0. They showed that this captured the phenomenon observed in several group-theoretic problems that are known to have non-computable instances while being simple to solve for every case encountered in practice; for instance, they demonstrated that for any $G$ in an extremely large class of groups, the word problem for $G$ has linear-time generic-case complexity. Myasnikov, in collaboration with Hamkins, went on to apply these ideas to Turing's halting problem \cite{genericHalting}, and proved that (for reasons having to do with the prevalence of trivially halting or trivially non-halting programs in many models of computation) the halting problem is ``generic-case decidable'' in said models. This was later refined by Rybalov \cite{stronglyGenericHalting}, who proved that the halting problem is \emph{not} ``strongly generic-case decidable'' (that is, decidable modulo sets with partial density converging to 0 exponentially fast); this proof, by contrast, is valid for all Turing-machine models of computation.

Jockusch and Schupp \cite{JSgc} have since defined and begun the study of the computability theory corresponding to generic-case complexity, and more generally the relations between asymptotic density and computability. Their work has been further developed in collaboration with Downey \cite{DJSdensity} and McNicholl \cite{ershovDensity}, and refined in specific cases by Igusa \cite{igusaNoMinimalPair} and Bienvenu, H\"olzl, and Day. \cite{absoluteUndecidables}

However, we return the focus to the notion of negligibility, since one would expect such a definition to have interesting ties to classical computability theory. For one, a negligible set might be said to be ``small'', ``sparse'', or even ``thin''. Such ``thinness'' properties have historically been of great interest in computability; they were the focus of Post's program \cite{post1944}, the first attempt to construct an incomplete \ce{} set, and have since proven to be of interest for unrelated reasons.

Negligibility (in the sense of asymptotic density) is closed downwards under the subset relation; any subset of a negligible set is itself negligible. It seems natural that it should be in the same family as the classical immunity properties, which provide the unifying computability-theoretic model for ``thinness''. However, negligibility does not lend itself to the same analysis that we apply to immunity. Choosing an alternate coding for the parameters of a membership problem is equivalent to applying a computable permutation to the underlying set, which can dramatically alter its asymptotic density. The most extreme example comes when we consider the class of infinite, co-infinite computable sets; the resulting consequences for \ce{} and co-\ce{} sets are essential to the remainder of this paper. We will need one standard definition of computability theory to incorporate a result of Downey, Jockusch, and Schupp: we say that a real $a$ is left-$\Sigma^0_2$ (left-$\Pi^0_2$) if its left cut is $\Sigma^0_2$ ($\Pi^0_2$).

\begin{prop}\label{prop:computableDensity}
Suppose $A$ is an infinite, co-infinite computable set. For any left-$\Sigma^0_2$ real $a$ and any left-$\Pi^0_2$ real $b$ with $0\le a\le b\le 1$, there is a computable permutation $\pi\!:\omega\to\omega$ such that $\pi(A)$ has lower density $a$ and upper density $b$.
\end{prop}
\begin{proof}
We note first that there is an infinite, co-infinite computable set $B$ with lower density $a$ and upper density $b$. In fact, this is nearly a theorem of Downey, Jockusch, and Schupp \cite{DJSdensity}, which states that for any $a$ and $b$ meeting our preconditions, there is a computable set $B$ with lower density $a$ and upper density $b$. Unless $a=b=0$ or $a=b=1$, this already ensures that $B$ is both infinite and co-infinite; if considering one of these cases, let $B$ be the set of perfect squares or its complement, respectively.

Since the infinite, co-infinite computable sets form an orbit under computable permutations, there is a computable permutation $\pi\!:\omega\to\omega$ such that $\pi(A)=B$; therefore, $\pi(A)$ has lower density $a$ and upper density $b$.
\end{proof}

\begin{cor}\label{cor:ceDensity}
If $A$ is infinite and \ce{}, there is a computable permutation \mbox{$\pi\!:\omega\to\omega$} such that $\pi(A)$ has density 1.
\end{cor}
\begin{proof}
Since $A$ is infinite and \ce{}, $A$ has an infinite (and co-infinite) computable subset $B$. By Proposition~\ref{prop:computableDensity}, there is a computable permutation $\pi\!:\omega\to\omega$ such that $\pi(B)$ has density 1. Since $B\subseteq A$, $\pi(B)\subseteq\pi(A)$, so $\pi(A)$ must also have density 1.
\end{proof}

\begin{cor}\label{cor:coceDensity}
If $A$ is co-infinite and co-\ce{}, there is a computable permutation $\pi\!:\omega\to\omega$ such that $\pi(A)$ has density 0.
\end{cor}

Since any infinite \ce{} set has density 1 under some computable permutation, any problem that is decidable on some infinite \ce{} subset of $\omega$ is in fact generic-case decidable if we choose the ``correct'' coding of the input. The corresponding coding is usually highly artificial, having little to do with the problem at hand.

In short, due to the sensitivity of asymptotic density to computable permutation, generic-case computability is sensitive to the coding we choose for a given problem. As some of the great strengths of Turing computability come from its invariance under choice of coding, we might hope to strengthen generic-case computability in such a way as to recover this invariance. To do so, we need to develop a stronger concept of negligibility, considering not only the upper and lower densities of a set, but those of all its images under computable permutations of $\omega$.

In Section~\ref{sec:intrinsicDensity}, we follow this approach and obtain a new pseudo-measure, \emph{intrinsic density}, which is invariant under computable permutations of $\omega$. We discuss various classes of sets that have intrinsic density, including the 1-random sets, which provide the foundation for our investigations in the rest of this paper.

For the remainder of the paper, we turn our focus to the new properties of intrinsic density. In Section~\ref{sec:idImmunity}, we begin with intrinsic density 0, the natural notion of being \emph{intrinsically negligible}, discussing it in the context of classical computability theory. In fact, intrinsic density 0 is an immunity property, fitting naturally into the hierarchy between immunity and cohesiveness, and we determine its place in the hierarchy for both unrestricted and $\Delta^0_2$ sets. In order to complete our description, we improve on a result of Downey, Jockusch, and Schupp \cite{DJSdensity}, constructing a strongly hyperhyperimmune set with upper density at least $1-\eps$ below $\0'$.

In Section~\ref{sec:idRandomness}, we reflect on the relation between intrinsic density and randomness, and the connection it provides between classical computability and randomness. In fact, intrinsic density provides a continuum from immunity to stochasticity, as any intrinsic density from the range $(0,1)$ is a version of stochasticity (modulo a fixed bias), while intrinsic density 0 is a form of immunity (as discussed in the previous Section). In fact, this correspondence can be reversed to extract various strengthenings of asymptotic density from the assorted notions of stochasticity --- some of which may prove fruitful topics of interest for future research.

Lastly, in Section~\ref{sec:intrinsicComputability}, we return to our motivating problem: the task of strengthening generic-case computability. After discussing some additional reasons for considering computably invariant notions of generic-case computability, we propose four such definitions, varying in degree of uniformity. All are strictly weaker than ordinary Turing computability, but even the weakest of our notions does not consider the halting problem (or, in fact, any nontrivial index set) to be computable.

In the remainder of the introduction, we collect notation and definitions that will be used for the rest of this paper. We will denote the $e$-th partial computable function by $\varphi_e$.

We will routinely identify a set $S\subseteq\omega$ with its characteristic function, $S(n)$, and also with the infinite binary sequence defining its characteristic function, $\set{S(n)}_{n\in\omega}$. By $S\upto n$, we mean either $S\cap\left[0,n\right)$ or the string consisting of the initial $n$ bits of the infinite sequence; which notation we are using at a given moment will be made clear by context. Two sets $S$ and $T$ have \emph{finite difference}, denoted $S\fineq T$, if $S(n)=T(n)$ for all sufficiently large $n$.

Given two finite binary strings $v$ and $w$, we say $v$ is a \emph{prefix} of $w$, denoted $v\preceq w$, if there is a binary string $x$ such that the concatenation of $v$ followed by $x$ is $w$ (\ie{}, $vx=w$); this definition extends to infinite binary sequences $w$ in the natural way.

The prefix-free Kolmogorov complexity of a binary string $w$ is denoted as $K(w)$; we refer to Downey and Hirschfeldt \cite{dhBook} or Nies \cite{niesCaR} for the details of its definition and properties, but note that it does relativize: we can consider the prefix-free Kolmogorov complexity of $w$ with respect to $A$, denoted $K^A(w)$. Both of these books also provide many equivalent characterizations of a 1-random set; for this paper, we will take the characterization in terms of the Kolmogorov complexity of initial segments as our definition. A set $S\subseteq\omega$ is \emph{1-random} if there is some constant $c$ such that $K(S\upto n)\ge n-c$ for all $n$. This definition inherits a natural relativization from prefix-free complexity: $S$ is 1-random relative to $A$ if there is some $c$ such that $K^A(S\upto n)\ge n-c$ for all $n$.

A set $S\subseteq\omega$ is \emph{1-generic} if, for every \ce{} set $X$ of finite binary strings, there is some initial segment $\sigma\prec S$ such that either $\sigma\in X$ or $\sigma\not\preceq\tau$ for every $\tau\in X$.

\section{Intrinsic density}
\label{sec:intrinsicDensity}

\begin{defn}
Let $S\subseteq\omega$. The \emph{absolute lower density} $\ald(S)$ of $S$ is
\begin{equation*}
\ald(S)\defeq\inf_{\pi}{\underline{\rho}(\pi(S))},
\end{equation*}
and the \emph{absolute upper density} $\aud(S)$ of $S$ is
\begin{equation*}
\aud(S)\defeq\sup_{\pi}{\overline{\rho}(\pi(S))},
\end{equation*}
where $\pi\!:\omega\to\omega$ is taken to vary over the set of computable permutations.

If the absolute upper and lower densities are equal, then we say that $S$ has \emph{intrinsic (asymptotic) density} $\id(S)$, where
\begin{equation*}
\id(S)\defeq\aud(S)=\ald(S).
\end{equation*}
In this case, not only does $S$ have a density, but its density is fixed under all computable permutations. We can develop analogous definitions for lower and upper densities; if $\ld(S)=\ld(\pi(S))$ for all computable permutations $\pi\!:\omega\to\omega$, we say that $S$ has \emph{intrinsic lower density} $\ld(S)$, and similarly for upper density.

If a set has intrinsic density 0, we say it is \emph{intrinsically negligible}.
\end{defn}

By Proposition~\ref{prop:computableDensity}, all infinite, co-infinite computable sets have absolute lower density 0 and absolute upper density 1. Thus, they are ``as far as possible'' from having an intrinsic density, at least in the sense that, under computable permutations, their densities range as widely as possible.

However, some might argue that 1-generic sets are further from having an intrinsic density than computable sets. It is simple to show that all 1-generic sets have lower density 0 and upper density 1. Since the class of 1-generic sets is closed under computable permutation, we can conclude that all 1-generic sets in fact have \emph{intrinsic} lower density 0 and \emph{intrinsic} upper density 1. This puts them ``as far as possible'' from having an intrinsic density, in the sense that no computable permutation can bring their upper and lower densities together.

For the rest of our work in this paper, we will focus primarily on sets that have an intrinsic density, rather than classes of sets that do not. With a few examples, we begin to establish the connections between intrinsic density and other computability-theoretic properties, and (particularly in discussing sets with intrinsic density strictly between 0 and 1) lay the groundwork for our later results.

We start with the r-cohesive and r-maximal sets. Recall that an infinite set $C$ is \emph{r-cohesive} if there is no computable set $R$ such that $R\cap C$ and $\overline{R}\cap C$ are both infinite, while a \ce{} set $C$ is \emph{r-maximal} if its complement is r-cohesive.

\begin{thm}[Jockusch, private correspondence]\label{thm:rCohesiveDensity}
Every r-cohesive set has intrinsic density 0.
\end{thm}
\begin{proof}
We note that if a set $C$ is r-cohesive, then its image under any computable permutation of $\omega$ is also r-cohesive; it thus suffices to prove that every r-cohesive set has density 0.

If we have a finite computable partition of $\omega$ (\ie{}, $\set{R_0,R_1,\ldots,R_{n-1}}$ computable and pairwise disjoint, with union $\omega$), $C$ must have finite intersection with all but one of these $R_i$, say $R_j$. By the finite subadditivity of upper density, $\ud(C)\le 0+\ud(C\cap R_j)\le\ud(R_j)$. If we take $R_i=\setbuild*{kn+i}{k\in\omega}$, we have that $\rho(R_i)=\frac{1}{n}$, so $\ud(C)\le\frac{1}{n}$. Since $n$ was an arbitrary natural number, the upper density of $C$ must be 0, so $\rho(C)=0$.
\end{proof}

\begin{cor}
Every r-maximal set has intrinsic density 1.
\end{cor}

However, sets of intermediate intrinsic density (strictly between 0 and 1) provide a more versatile basis for further investigation; as such, the 1-random sets will be essential to certain constructions later in this paper.

\begin{prop}\label{prop:randomDensity}
Every 1-random set has intrinsic density $\frac{1}{2}$.
\end{prop}
\begin{proof}
Any 1-random set obeys the Law of Large Numbers, in the sense that it has density $\frac{1}{2}$. \cite{niesCaR}{~(Prop.~3.2.13)} Since the class of 1-random sets is closed under computable permutations of $\omega$, every 1-random set has intrinsic density~$\frac{1}{2}$.
\end{proof}

We can use 1-randoms to construct sets of other intermediate intrinsic densities as well, by means of the following lemma and its corollaries.

\begin{lemma}\label{lem:randomThinning}
If $A$ has density $d$, and $B$ is 1-random relative to $A$, then $A\cap B$ has density $\frac{d}{2}$.
\end{lemma}
\begin{proof}
Interpreting $B$ as a binary sequence, consider the $A$-computable subsequence $\what{B}$ selected by the rule ``If $n\in A$, select $B(n)$.'' Since $B$ is 1-random relative to $A$, we see that $\what{B}$ must be an unbiased sequence; in other words, $\rho(\what{B})=\frac{1}{2}$.

However, by the definition of $\what{B}$ and asymptotic density,
\begin{equation*}
\rho(\what{B})=\lim_{n\to\infty}{\frac{\abs{\paren*{A\cap B}\upto n}}{\abs{A\upto n}}},
\end{equation*}
so
\begin{align*}
\rho(A)\rho(\what{B})
&=\paren*{\lim_{n\to\infty}{\frac{\abs{A\upto n}}{n}}}\paren*{\lim_{n\to\infty}{\frac{\abs{\paren*{A\cap B}\upto n}}{\abs{A\upto n}}}}\\
&=\lim_{n\to\infty}{\frac{\abs{\paren*{A\cap B}\upto n}}{n}}\\
&=\rho(A\cap B).
\end{align*}
Therefore, $A\cap B$ has density $\rho(A)\rho(\what{B})=\frac{d}{2}$.
\end{proof}

Since being 1-random is invariant under computable permutation, we obtain one more pair of corollaries:
\begin{cor}
If $A$ has intrinsic density $d$, and $B$ is 1-random relative to $A$, then $A\cap B$ has intrinsic density $\frac{d}{2}$.
\end{cor}

\begin{cor}\label{cor:randomIntersectionDensity}
If $\set{A_1,\ldots,A_n}$ are mutually relatively 1-random sets (\ie{}, each set is 1-random relative to the join of the others), then $\bigcap_{1\le i\le n}{A_i}$ has intrinsic density $2^{-n}$.
\end{cor}


Having established a few tools to use in controlling the intrinsic density of sets (in this paper, largely useful for the construction of counterexamples), we can now proceed to consider intrinsic density in a broader context.

\section{Intrinsic density and immunity}
\label{sec:idImmunity}

As discussed in the Introduction, asymptotic density was defined as a substitute for a probability measure on a countable space. Its use in generic-case computability (and other topics) is in defining a density-0 set to be negligible, in the sense that its elements are eventually scarce. This provides one of the more practical notions of a ``small'' or ``thin'' subset of the integers, in some senses  more natural than asserting that a set has no infinite \ce{} subset (\ie{}, is \emph{immune}).

Unfortunately, having density 0 is not computably invariant. From the perspective of computability theory, set properties that vary under computable permutation have limited applications. By addressing this one issue, having intrinsic density 0 proves to be more powerful; for example, any infinite set having intrinsic density 0 (or, in fact, any intrinsic lower density other than 1) must be immune.

\begin{prop}
Any infinite non-immune set has density 1 under some computable permutation.
\end{prop}
\begin{proof}
This immediately follows from Corollary~\ref{cor:ceDensity}. If $S$ is infinite and not immune, it contains an infinite \ce{} subset $A$. By Corollary~\ref{cor:ceDensity}, $\pi(A)$ has density 1 for some computable permutation $\pi$. Since $S\supseteq A$, $\pi(S)\supseteq\pi(A)$, so $\pi(S)$ must also have density 1.
\end{proof}

\begin{cor}\label{cor:id0Immune}
Any infinite set with intrinsic lower density 0, and hence, any infinite set with intrinsic density 0, is immune.
\end{cor}

It is clear that the upper density of a set bounds the upper density of any of its subsets, so intrinsic density 0 is closed downwards under the subset relation. Since having intrinsic density 0 is a computably invariant property, closed under subsets, and implies immunity, intrinsic density 0 (here abbreviated id0) is a natural new immunity property, describing a strong notion of ``thinness''.

We therefore seek to determine its relation to the classical immunity properties:

\begin{defn}
A \ce{} list of pairwise disjoint finite sets $\set{D_i}$ (indexed as finite sets, so that the sets $D_i$ and the function $i\mapsto\max{D_i}$ are computable) is said to be a \emph{strong array}. Similarly, a uniformly \ce{} list of pairwise disjoint \ce{} sets $\set{U_i}$ is a \emph{weak array}. There appears to be no standard name for the intermediate concept, which we here term a \emph{computable array}: a \ce{} list of pairwise disjoint computable sets $\set{C_i}$ (indexed appropriately), with union $\bigcup{C_i}$ also computable.

An infinite set $A$ is \emph{hyperimmune} (sometimes abbreviated h-immune) if for every strong array $\set{D_i}$, there is some $j$ such that $A\cap D_j=\emptyset$; in this case, we say that $\set{D_i}$ fails to meet $A$. Similarly, $A$ is said to be \emph{strongly hyperimmune} (sh-immune) if no computable array meets $A$, and \emph{strongly hyperhyperimmune} (shh-immune) if no weak array meets $A$. In a slight generalization, we say $A$ is \emph{finitely strong hyperimmune} (fsh-immune) if no computable array $\set{C_i}$ with all $C_i$ finite meets $A$, and \emph{hyperhyperimmune} if no weak array $\set{U_i}$ with all $U_i$ finite meets $A$.

It was quickly noted that a set $A$ is hyperimmune if{}f no computable function dominates its principal function $p_A(n)\defeq\paren*{\mu x}\bracket*{\abs{S\upto x}\ge n}$; that is, for all computable functions $f$, we have $p_A(n)\ge f(n)$ infinitely often. Strengthening this, we say that an infinite set $A$ is \emph{dense immune} if its principal function dominates all computable functions: for all computable functions $f$ and all sufficiently large $n$, we have $p_A(n)\ge f(n)$.

An infinite set $A$ is \emph{cohesive} if, for all \ce{} sets $U_i$, either $A\cap U_i$ or $A\cap\overline{U_i}$ is finite. We can weaken this in a few standard ways: $A$ is \emph{r-cohesive} if the same property holds for computable sets $C_i$, or \emph{quasicohesive} (q-cohesive) if $A$ is a finite union of cohesive sets.
\end{defn}

These properties are organized in a natural hierarchy of implication, shown as Figure~\ref{fig:immunityImplications}. Chapter XI.1 of Soare \cite{soareREsets} provides an excellent reference for this hierarchy (though focused on co-\ce{} sets). We note that in the general case, the lack of implication between cohesiveness and dense immunity is witnessed by the existence of a non-high cohesive set, as constructed by Jockusch and Stephan \cite{nonhighCohesive}. Also, shh-immunity and hh-immunity are distinguishable in the general case, but by a remark of Cooper \cite{cooperHHimmune}, are equivalent for $\Delta^0_2$ sets (and thus for co-\ce{} sets).

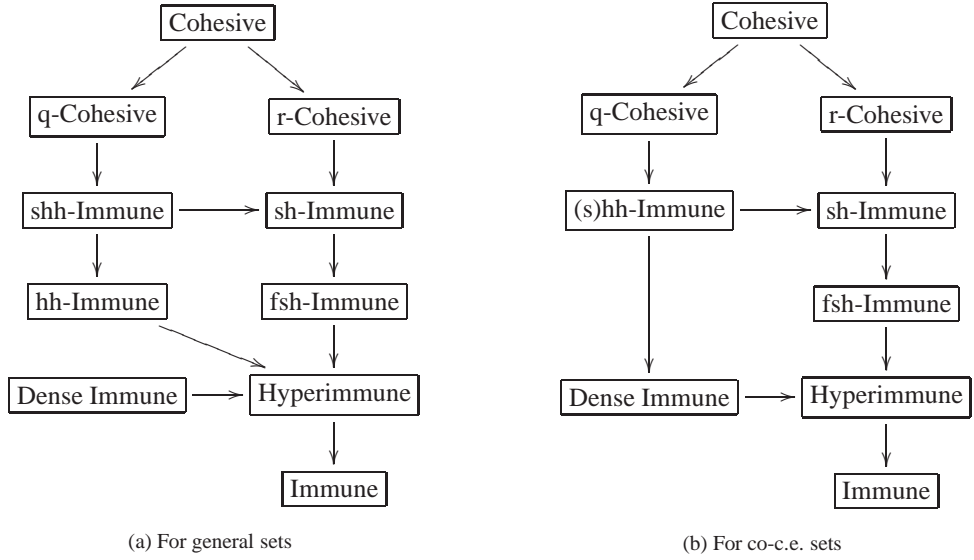
\begin{figure}
\centering
\hfill%
\begin{subfigure}{.35\linewidth}
\centering
\resizebox{\linewidth}{!}{
\xymatrix@R-8pt@C=-15pt{
	& & *+[F]+\hbox{Cohesive} \ar[dl] \ar[dr] &\\
	& *+[F]+\hbox{q-Cohesive} \ar[d] & & *+[F]+\hbox{r-Cohesive} \ar[d]\\
	& *+[F]+\hbox{shh-Immune} \ar[d] \ar[rr] & & *+[F]+\hbox{sh-Immune} \ar[d]\\
	& *+[F]+\hbox{hh-Immune} \ar[drr] & & *+[F]+\hbox{fsh-Immune} \ar[d]\\
	& *+[F]+\hbox{Dense Immune} \ar[rr] & & *+[F]+\hbox{Hyperimmune} \ar[d]\\
	& & & *+[F]+\hbox{Immune}
	}
}
\subcaption{For general sets}
\end{subfigure}%
\hfill%
\begin{subfigure}{.35\linewidth}
\centering
\resizebox{\linewidth}{!}{
\xymatrix@R-8pt@C=-15pt{
	& *+[F]+\hbox{Cohesive} \ar[dl] \ar[dr] & \\
	*+[F]+\hbox{q-Cohesive} \ar[d] & & *+[F]+\hbox{r-Cohesive} \ar[d]\\
	*+[F]+\hbox{(s)hh-Immune} \ar[dd] \ar[rr] & & *+[F]+\hbox{sh-Immune} \ar[d]\\
	& & *+[F]+\hbox{fsh-Immune} \ar[d]\\
	*+[F]+\hbox{Dense Immune} \ar[rr] & & *+[F]+\hbox{Hyperimmune} \ar[d]\\
	& & *+[F]+\hbox{Immune}
	}
}
\subcaption{For co-\ce{} sets}
\end{subfigure}
\hspace*{\fill}%
\caption{The graphs of the implications between the classical immunity properties; for $\Delta^0_2$ sets, the implications are the same as in the general case, except that shh-immunity and hh-immunity become equivalent. All implications are strict, and any not shown (excepting those implied by transitivity) are false.}
\label{fig:immunityImplications}
\end{figure}

By Theorem~\ref{thm:rCohesiveDensity}, r-cohesiveness implies intrinsic density 0. This has a simple corollary, included here for completeness:

\begin{cor}
Every quasi-cohesive set has intrinsic density 0.
\end{cor}
\begin{proof}
As a finite union of cohesive (and thus r-cohesive) sets, any quasi-cohesive set $Q$ is a finite union of sets of intrinsic density 0. Since density is finitely subadditive, $Q$ must also have intrinsic density 0.
\end{proof}

It is slightly more interesting to note that dense immunity also implies intrinsic density 0. To show this, we note that dense immunity is computably invariant, and that a certain domination property is equivalent to having density 0.

\begin{lemma}
For any infinite set $S\subseteq\omega$,
\begin{equation*}
\ud(S)\defeq\limsup_{n\to\infty}{\frac{\abs{S\upto n}}{n}}=\limsup_{n\to\infty}{\frac{n}{p_S(n)}},
\end{equation*}
where $p_S\defeq\paren{\mu x}\bracket*{\abs{S\upto x}\ge n}$ is the principal function of $S$.
\end{lemma}
\begin{proof}
Consider the sequence $\set{a_n}=\set{\frac{\abs{S\,\upto\,n}}{n}}_{n\in\omega}$. We note that $\set{b_n}=\set{\frac{n}{p_S(n)}}_{n\in\omega}$ is an infinite subsequence --- in fact, $b_n=a_{p_S(n)}$ for all $n\in\omega$ --- so
\begin{equation*}
\limsup_{n\to\infty}{b_n}\le\limsup_{n\to\infty}{a_n}=\ud(S).
\end{equation*}
The sequence $a_n$ may increase only at $n$ in the image of $p_S$ (and thus at points also in the subsequence $b_n$), so these limits must be equal.
\end{proof}

\begin{prop}\label{prop:density0Domination}
An infinite set $S\subseteq\omega$ has density 0 if{}f its principal function dominates all linear functions. (In standard asymptotic [``Big-O''] notation, $S$ has density 0 if{}f $p_S(n)\in\omega(n)$.)
\end{prop}
\begin{proof}
$S$ has density 0 if{}f $\ud(S)=0$, and by the preceding lemma,
\begin{equation*}
\ud(S)=\limsup_{n\to\infty}{\frac{n}{p_S(n)}}.
\end{equation*}
However, for all $k>0$, $\limsup_{n\to\infty}{\frac{n}{p_S(n)}}<\frac{1}{k}$ if{}f for all $c\in\R$, $p_S(n)$ dominates $kn+c$; therefore, $\ud(S)=0$ if{}f $p_S(n)$ dominates $kn+c$ for all $k>0$ and $c\in\R$.
\end{proof}

\begin{prop}\label{prop:denseImmuneInvariant}
If the set $A$ is dense immune, $\pi(A)$ is also dense immune for any computable permutation $\pi$.
\end{prop}
\begin{proof}
Let $\pi$ be a computable permutation, and consider a computable function $f$. We define
\begin{equation*}
\what{f}(n)=1+\max_{x\in\left[0,f(n)\right)}{\pi^{-1}(x)}.
\end{equation*}
Since $\what{f}$ is a computable function, the principal function of $A$ must dominate $\what{f}$; that is, $p_A(n)>\what{f}(n)$ for all but finitely many $n$. In other words, for all sufficiently large $n$, there are fewer than $n$ elements of $A$ less than $\what{f}(n)$.

However, by construction of $\what{f}$, every element of $\pi(A)$ less than $f(n)$ must come from an element of $A$ below $\what{f}(n)$. Since for almost all $n$, there are fewer than $n$ elements of $A$ below $\what{f}(n)$, we see that $p_{\pi(A)}$ dominates $f$. Since both $f$ and $\pi$ were arbitrary, every computable permutation of $A$ is dense immune.
\end{proof}

\begin{cor}
If $S\subseteq\omega$ is dense immune, $S$ has intrinsic density 0.
\end{cor}
\begin{proof}
By Proposition~\ref{prop:denseImmuneInvariant}, dense immunity is a computably invariant property. It therefore suffices to show that dense immunity implies density 0. However, this is an immediate consequence of Proposition~\ref{prop:density0Domination}, as all linear functions are bounded by computable functions, and so are dominated by the principal function of any dense immune set.
\end{proof}

None of the remaining standard immunity properties imply intrinsic density 0. In fact, as demonstrated by Downey, Jockusch, and Schupp \cite{DJSdensity}, for every $\eps>0$, there is a strongly hyperhyperimmune set with upper density at least $1-\eps$ (though none have upper density 1), constructed by Mathias forcing. We here extend their result, using a direct $\0'$-computable construction, to show that we may assume these sets to be $\Delta^0_2$. Of course, per the aforementioned remark of Cooper \cite{cooperHHimmune}, any $\Delta^0_2$ hh-immune set is in fact shh-immune; this apparently does not simplify the argument required, so we will make no use of this fact.

\begin{thm}\label{thm:hhImmune}
For all $\eps>0$, there is a $\Delta^0_2$ (s)hh-immune set $A$ with upper density at least $1-\eps$.
\end{thm}
\begin{proof}
We assume that $\eps$ is rational; if not, we can replace it by any rational less than $\eps$. We will construct $A$ as a $\0'$-computable set, consulting $\0'$ as an oracle during our otherwise-computable construction.

We work to satisfy the requirements:
\begin{align*}
P_e:\ &\paren*{\exists n>e}\bracket*{\rho_n(A)\ge 1-\eps},\\
N_e:\ &\paren*{\forall i,j}\bracket*{\paren*{i\ne j}\implies\paren*{U_{e,i}\cap U_{e,j}=\emptyset}}\implies\paren*{\exists k_e}\bracket*{A\cap U_{e,k_e}=\emptyset},
\end{align*}
where $\set{U_{e,i}}$ is a listing of the uniformly \ce{} sequences of sets. (In other words, every uniformly \ce{} sequence of sets is of the form $\set{U_{e,i}}_{i\in\omega}$ for some $e$.)

The negative requirements $N_e$ together assert that every weak array fails to meet $A$; this is the definition of shh-immunity.

Our negative requirements should, in principle, be simple to satisfy. We simply omit a set from each weak array of small lower density, thus leaving us with a set $A$ with high upper density. There are slight complications in ensuring that taking all of these requirements still cannot force $A$'s upper density to fall below $1-\eps$, but these are easily addressed. After all, at any given point $n$, only $q$ disjoint sets can have partial density exceeding $\frac{1}{q}$; therefore, for any weak array $\set{U_{e,i}}_{i=0}^{\infty}$, there must be some $U_{e,k}$ with lower density less than $2^{-e}\eps$.

The trouble comes in attempting to identify the $U_{e,k}$ in question. $\0'$ is incapable of resolving whether a \ce{} set has lower density below some bound; in fact, this problem is $\Sigma^0_2$-hard, as it would allow us to distinguish finite \ce{} sets from total \ce{} sets. Therefore, we cannot search directly for such a $U_k$ in our weak array, and must use more indirect methods.

Towards this end, we will make heavy use of the following sentence, for varying values of $e$, $n$, and $r$:
\begin{equation}\label{eqn:maxTuple}
\begin{split}
(\exists x_n>x_{n-1}&>\cdots>x_1\ge r)\paren*{\exists s}\paren*{\exists t>s}\\
&\bracket*{t\in C_{e-1,s}\land\paren*{\forall i\le n}\bracket*{\rho_t(U_{e,x_i})\ge d_e}}.
\end{split}\tag{*}
\end{equation}
$C_{e-1,s}$ will be defined in the course of our construction, but is \ce{}. The $U_{e,i}$'s are taken from our listing of uniformly \ce{} sequences of sets. As all sets involved are \ce{}, and we only ask whether a \ce{} set has partial density exceeding some lower bound, \eqref{eqn:maxTuple} is a $\Sigma^0_1$ sentence, and thus decidable by $\0'$.

Putting \eqref{eqn:maxTuple} in context, we understand it to say that there are $n$ elements of our weak array $\set{U_{e,i}}$, not including any with index less than $r$, which all have high partial density (exceeding $d_e$) at a single point $t>s$, where $t$ is chosen from some \ce{} set $C_{e-1}$ of possibilities. Clearly, this sentence is vacuously true for $n=0$, and (presuming our $U_{e,i}$'s to be disjoint) necessarily false for $n>\frac{1}{d_e}$. Therefore, for any fixed $r$ and $e$, the maximum $n$ for which this sentence holds is computable in $\0'$ by a simple bounded search; let us refer to it as $N$. If we have $N$ such elements of a weak array, we refer to them together as a maximal tuple for that array under the conditions $C_{e-1}$, $r$, and $e$.

If we have a maximal tuple for our weak array, and $U_{e,k}$ is not among its members, then we know that $\rho_t(U_{e,k})<d_e$ for some $t\in C_{e-1}$ with $t>s$. Otherwise, \eqref{eqn:maxTuple} would be satisfied with $n=N+1$. This will be our primary tool for controlling the lower density of the sets we omit as we build $A$ to avoid meeting the weak array $\set{U_{e,i}}_{i=0}^{\infty}$.

\textbf{\emph{Organization}}:

\nopagebreak

As we combine multiple negative requirements, we allow finite injury of each negative requirement by higher-priority requirements, though never revoking any previous decisions as to whether $s\in A$. We activate the requirements in order of decreasing priority, activating at most one at each stage. At stage $s$, each active requirement independently decides whether to allow $s$ into $A$; we put $s$ into $A$ if none of these requirements object.  For convenience, we denote $A\upto s$ by $A_s$.

Each requirement $N_e$ will maintain a \ce{} set $S_{e,s}$ of elements such that, if $A\cap S_{e,s}=\emptyset$, then $N_e$ will be satisfied. For internal reference, we will also keep track of $k_{e,s}$, which determines which element of the weak array $\set{U_{e,i}}$ we are actually restricting out of $A$. Lastly, we will maintain a \ce{} set $C_{e,s}$ of locations where the partial density of $U_{e,k_{e,s}}$ is known to be strictly less than $2^{-e-3}\eps$, while guaranteeing that $C_{e,s}\subseteq C_{e-1,s}$ at all stages $s$.

$N_e$'s basic goal is to prevent the weak array $\set{U_{e,i}}$ from meeting $A$, while ensuring that the lower density of its restricted set does not exceed $d_e=2^{-e-3}\eps$. To do so, $N_e$ will repeatedly consult $\0'$ regarding \eqref{eqn:maxTuple}. In context, we can now see that we choose $t\in C_{e-1,s}$ to ensure that the density of the set we omit for $N_e$ falls below $d_e$ at the same time as the densities of the previously-chosen sets fall below their critical values; this will make certain that the density of our set $A$ rises above its goal of $1-\eps$.

\textbf{\emph{Module for $N_e$}}:

\nopagebreak

\emph{On activation at stage $s$:} We first consult $\0'$, asking whether the sets $\set{U_{e,i}}$ are in fact pairwise disjoint (\ie{}, $\set{U_{e,i}}$ is a weak array). If not, then $N_e$ is trivially satisfied. In this case, $N_e$ will never restrict anything out of $A$; it simply maintains $S_{e,t}=\emptyset$ and $C_{e,t}=C_{e-1,t}$ at all stages $t\ge s$, while voting to allow all elements into $A$.

If $\set{U_{e,i}}$ is a weak array, we define $k_{e,s-1}$ to be the least $k$ such that $U_{e,k}\upto s=\emptyset$. We then set $S_{e,s-1}=U_{e,k_{e,s-1}}$, and let $C_{e,s-1}=\emptyset$, as we do not yet know of any locations where $\rho_t(S_{e,s})<d_e$.

\emph{At stage $s$:} We assume that $A_s=A\upto s$ has already been determined, and consider only whether to allow $s$ into $A$. Before making this decision, we must first determine whether we can still believe that we can restrict $S_{e,s-1}$ out of $A$ while keeping $\ud(A)$ close to 1. In fact, we want to verify that $S_{e,s-1}$ will again appear to have partial density less than $d_e$ at some point $t>s$ where the partial density of $S_{i,s}$ (for all $i<e$) is also small.

If $C_{e,s-1}\cap\paren*{s,\infty}\ne\emptyset$ (a $\0'$-computable question), we have already verified this at some previous stage. We simply define $S_{e,s}=S_{e,s-1}$, set $k_{e,s}=k_{e,s-1}$, and let $C_{e,s}=C_{e,s-1}$. We then allow $s$ into $A$ if{}f $s\not\in S_{e,s}$.

If $C_{e,s-1}\subseteq\bracket*{0,s}$, though, we must attempt to verify that the partial density of $S_{e,s-1}$ will fall below $d_e$ at some point in the future. We know that $\0'$ cannot answer this question directly, as it cannot determine whether a \ce{} set will ever have partial density less than some critical value. We instead use \eqref{eqn:maxTuple} to attack from a different angle. We will need to reference $k_{e,s-1}$ several times in the remainder of the procedure; for simplicity's sake, we will abbreviate it by $k=k_{e,s-1}$.

We first determine $n_{e,s}$, the greatest value of $n$ for which \eqref{eqn:maxTuple} holds with $r=k$. Since at most $\floor{1/d_e}$ disjoint sets can have partial density exceeding $d_e$ at the same location $t$, this is a bounded search on a parameter of a $\Sigma^0_1$ statement; thus, $\0'$ suffices to compute $n_{e,s}$.

We then ask $\0'$ whether \eqref{eqn:maxTuple} holds with $n=n_{e,s}$ and $r=k+1$. If so, then we have a maximal tuple (within the array $\set{U_{e,i}}$ for $i>k$) in which every set has high partial density at the same point $t>s$. Since we cannot add $U_{e,k}$ to this collection, we must have $\rho_t(U_{e,k})<d_e$. We define $C_{e,s}$ to be the set of all $t>s$ for which there is such a collection (along with all $t\le s$ for which $\rho_t(S_{e,s})<d_e$), set $k_{e,s}=k_{e,s-1}$ and $S_{e,s}=S_{e,s-1}$, and allow $s$ into $A$ if{}f $s\not\in S_{e,s}$. If this case occurs immediately following injury or initialization of $N_e$ at stage $s-1$, we say that $s$ was a ``recovery stage'' for $N_e$; otherwise, we deactivate all lower-priority requirements $N_i$ ($i>e$), as $C_{e,s}$ has changed.

Otherwise, the \eqref{eqn:maxTuple} does not hold with $n=n_{e,s}$ and $x_1$ strictly greater than $k$. In this case, we have no way to verify that the density of $U_{e,k}$ again drops below $d_e$, and so consider $N_e$ to be injured. We vote to allow $s$ into $A$, and deactivate all lower-priority requirements $N_i$ ($i>e$). We then effectively reset our procedure for $N_e$; we define $k_{e,s}$ to be the least $k>k_{e,s-1}$ such that $U_{e,k}\cap\bracket*{0,s}=\emptyset$, set
\begin{equation*}
S_{e,s}=\paren*{S_{e,s-1}\cap\left[0,s\right)}\cup U_{e,k},
\end{equation*}
and let $C_{e,s}=\emptyset$.

\textbf{\emph{Verification of the basic module:}}

\nopagebreak

Suppose that the module for $N_e$ is at some point activated and never again deactivated (\ie{}, $C_{e-1,s}$ does not change at any later stage $s$). We assume that $\set{U_{e,i}}$ is in fact a weak array; if it is not, $N_e$ is trivially satisfied, $S_{e,s}=\emptyset$ has partial density identically 0 for all $s$, and $C_{e,s}=C_{e-1,s}$ does not change at any later stage $s$.

By the construction of $S_{e,s}$, we know that $S_e=\lim_{s\to\infty}{S_{e,s}}$ exists, and consists of all elements restricted out of $A$ by $N_e$. We will show that there is some stage $s_0$ at which $C_{e,s_0}$ is infinite, thus ensuring that $C_{e,s}$ will not change at any later stage and  preventing future injury to $N_e$. This will also guarantee that $S_e=S_{e,s_0}$ and $k_e=\lim_{s\to\infty}{k_{e,s}}=k_{e,s_0}$.

Given such an $s_0$, since $A$ does not intersect $S_{e,s_0}\supseteq U_{e,k_e}$, we have satisfied $N_e$. Furthermore, the partial density of $S_e$ approaches that of $U_{e,k_e}$, as the sets agree on all $x\ge s_0$; therefore, if the partial density of $U_{e,k_e}$ drops below $d_e$ infinitely often, the partial density of $S_e=S_{e,s_0}$ must be less than $2d_e$ at all but finitely many of the same points.

We note that the sequence $\set{n_{e,s}}$ is nonincreasing, as we monotonically reduce the set of witnesses for \eqref{eqn:maxTuple} at successive stages $s$. In fact, the sequence must decrease each time $C_{e,s}$ changes (except at recovery stages); this can only happen when we have run out of witnessing collections of size $n_{e,s-1}$. As for recovery stages, they can only occur immediately after initialization of $N_e$, or immediately after an injury to $N_e$; since injuries cause $n_{e,s}$ to decrease, any recovery stage is still associated with a corresponding decrease in $n_{e,s}$. Since for $n=0$, \eqref{eqn:maxTuple} is vacuously true, $n_{e,s}$ is always a non-negative integer and so cannot decrease infinitely often. Therefore, there must be some stage $s_0$ such that $C_{e,t}=C_{e,s_0}\ne\emptyset$ for all $t>s_0$, which is only possible if $C_{e,s_0}$ is infinite.

Lastly, the module must force $A\cap S_e=\emptyset$. Whenever we choose a new $k_{e,s}$, we always choose a value $k$ such that $A_s\cap U_{e,k}=\emptyset$, and redefine $S_{e,s}$ accordingly to remain disjoint from $A_s$. When we keep the same $k$, we allow elements into $A$ if{}f they are not in $S_{e,s}$. Therefore, as long as $N_e$ is active, we are assured that $A_s\cap S_{e,s}=\emptyset$ for all $s$; since $S_e\upto s=S_{e,s}\upto s$, we will always have $A\cap S_e=\emptyset$.

\textbf{\emph{Construction of $A$:}}

\nopagebreak

At stage 0, begin by activating $N_0$.

At stage $s$, check whether $\rho_s(A)\ge 1-\eps$. If so, determine the highest-priority inactive requirement $N_e$. If $N_e$ was deactivated in stage $s-1$, do not activate any requirements; otherwise, activate $N_e$. (This delay in $N_e$'s reactivation ensures that $N_e$ is not activated during a recovery stage for $N_{e-1}$.)

Next, consult all active requirements in priority order. If any restrict $s$ out of $A$, we declare that $s\not\in A$; if all allow $s$ to enter $A$, we put $s$ into $A$.

\textbf{\emph{Verification:}}

\nopagebreak

Nothing can deactivate $N_0$, so $N_0$ is permanently activated. By the correctness of the basic module, if the module for $N_e$ is permanently activated, $N_e$ will be satisfied. Furthermore, there is some stage $s_0$ after which $C_{e,s}$ does not change, so that $N_{e+1}$ will never again be deactivated. Therefore, as long as there are infinitely many stages at which we activate some inactive requirement, every module will be permanently activated at some point, and thus every $N_e$ will be satisfied.

Suppose, towards a contradiction, that some requirement is never permanently activated. Let $N_e$ be the highest-priority such requirement, so that only modules $N_0$ through $N_{e-1}$ are permanently activated. We consider the construction at stage $s_0$, after the last such module has been permanently activated and $C_{e-1,s}$ has stopped changing (and is infinite).

At this stage, $N_e$ can never again be deactivated, so since $N_e$ is not permanently activated, $N_e$ must never again be activated. This can only be because the construction will never reach another stage where it activates the highest-priority inactive requirement; therefore, it must be that $A\cap\paren*{s,\infty}$ contains all elements except those in $S=\bigcup_{i<e}{S_i}$, and $\rho_t(A)<1-\eps$ for all $t>s$. For all sufficiently large $n$, we have that $\rho_n(A)<1-\eps$ implies $\rho_n(S)>\frac{1}{2}\eps$; thus, $\rho_n(S)>\frac{1}{2}\eps$ for all sufficiently large $n$.

However, $\rho_n(S)=\rho_n(\bigcup_{i<e}{S_i})\le\sum_{i<e}{\rho_n(S_i)}$. Recall that for all but finitely many $n\in C_{e-1}$, we have $\rho_n(S_i)<2d_e=2^{-n-2}\eps$, so
\begin{equation*}
\rho_n(S)\le\sum_{i<e}{2^{-i-2}\eps}<\frac{1}{2}\eps.
\end{equation*}
Since $C_{e-1}$ is infinite, this is a contradiction; therefore, all modules $N_e$ must be permanently activated eventually.

Finally, since every module is eventually activated, we must activate a new module infinitely often. This can only happen if $\rho_s(A)\ge 1-\eps$ infinitely often, so every requirement $P_e$ is also satisfied; $A$ must have upper density at least $1-\eps$.
\end{proof}

We have yet to consider implications in the other direction; what immunity properties are implied by intrinsic density 0? The first such result is simple; as established above in Corollary~\ref{cor:id0Immune}, intrinsic density 0 at least implies immunity for infinite sets.

On the other hand, we already know that hyperimmunity (even shh-immunity) does not imply intrinsic density 0. We can further prove that a set of intrinsic density 0 need not be hyperimmune; we can construct $\Delta^0_2$ counterexamples, and in fact will build a counterexample below every 1-random set.

\begin{thm}
For every 1-random set $R$, there is an infinite set $A\leT R$ with intrinsic density 0 that is not hyperimmune.
\end{thm}
\begin{proof}
Suppose that $K(R\upto n)\ge n-c$ for all $n$.

By van Lambalgen's Theorem \cite{vLaxioms}, given a 1-random set $R$, there exists a uniformly $R$-computable sequence of sets $\set{R_j}_{j\in\omega}$ that are mutually relatively 1-random. In fact, defining $\what{R}_j=\bigoplus_{i<j}{R_i}$, we have that
\begin{equation*}
K^{\what{R}_j}(R_j\upto\,n)\ge n-d_j
\end{equation*}
for all $n$, where $d_j$ is uniformly computable from $j$ and $c$; this can be shown by a simple inspection of a proof of van Lambalgen's Theorem.

Let $A_0=R_0$. Since $A_0$ is 1-random, it has intrinsic density $\frac{1}{2}$. Given $d_0$ and using the incompressibility of $R_0$, we can compute $k_0$ such that $\abs{R_0\upto\,k_0}\ge 1$, ensuring that $A_0\cap\left[0,k_0\right)\ne\emptyset$. Since $k_0$ is computable, we can use $\left[0,k_0\right)$ as the first partition in a weak array that will witness that the set we construct $A$ is not hyperimmune.

We then define
\begin{equation*}
A_1=A_0\cap\paren*{\left[0,k_0\right)\cup R_1}.  
\end{equation*}
Since $A_1\fineq R_0\cap R_1$, and $R_0$ and $R_1$ are mutually relatively 1-random, $A_1$ must have intrinsic density $\frac{1}{4}$ by Corollary~\ref{cor:randomIntersectionDensity}. Using $d_0$ and $d_1$ along with the incompressibility of $R_1$ (relative to $R_0$), we can compute $k_1$ such that $\abs{A_1\cap\left[k_0,k_1\right)}\ne\emptyset$.

Repeating this process, we see that $A=\bigcap_{j}{A_j}$ is computable in $R$, since $A\upto\,k_j=A_j\upto\,k_j$. For all $j$, we have that $A\cap\left[k_{j-1},k_{j}\right)\ne\emptyset$, so $A$ is infinite. Furthermore, since the $k_j$'s are uniformly computable from $c$, an integer, this partition of $\omega$ is in fact computable, demonstrating that $A$ is not hyperimmune.

Finally, $A\subseteq A_j$ for all $j$. Since $A_j\fineq\bigcap_{i\le j}{R_i}$, and the $R_i$'s are mutually relatively 1-random, $A_j$ has intrinsic density $2^{-j-1}$; therefore, $A$ must have intrinsic density 0.
\end{proof}

As a convenient side effect, this theorem immediately gives us some information on the Turing degrees of infinite sets with intrinsic density 0: such sets exist below every 1-random Turing degree, but cannot be computable. Among other things, this implies that there are infinite id0 sets in non-computable $\Delta^0_2$, low, and even hyperimmune-free degrees (where our construction of a non-hyperimmune $A$ becomes rather superfluous, though it at least ensures that $A$ is infinite).

It still eliminates any hopes we might have of further positive implications between the immunity properties and intrinsic density, as we will discuss in our summary below. By falling back to intrinsic lower density, we can recover one more positive implication, as shown by Jockusch in private correspondence.

\begin{thm}
Every hyperimmune set has intrinsic lower density 0.
\end{thm}
\begin{proof}
Since hyperimmunity is computably invariant, it suffices to show that every hyperimmune set has lower density 0.

Suppose $A$ is hyperimmune. Consider the strong array 
\begin{equation*}
D_n=\left[n!,(n+1)!\right).
\end{equation*}
Since $A$ is hyperimmune, $A\cap D_n=\emptyset$ for infinitely many $n$. For all such $n$, we have that $\abs{A\upto (n+1)!}\le n!$; therefore, $\rho_{(n+1)!}(A)\le\frac{1}{n}$. Since this occurs infinitely often, we conclude that $\underline{\rho}(A)=0$.
\end{proof}

The above results, along with earlier work \cite{JSgc} \cite{DJSdensity}, will suffice to disprove all other potential implications between intrinsic density 0, intrinsic lower density 0, and the standard immunity properties.

We first repeat, per Jockusch and Schupp \cite{JSgc}, that any 1-generic set has lower density 0 and upper density 1; since 1-genericity is computably invariant, 1-generics in fact have intrinsic lower density 0 and intrinsic upper density 1. Therefore, intrinsic lower density 0 does not imply intrinsic density 0, even for $\Delta^0_2$ sets.

In addition, all 1-random sets are immune; otherwise, there would be a 1-random $R$ with an infinite computable subset, which admits a trivial computable martingale that succeeds on $R$. Since 1-randoms have intrinsic density $\frac{1}{2}$, immunity does not imply intrinsic lower density 0, even for $\Delta^0_2$ sets.

Lastly, Theorem~\ref{thm:hhImmune} above demonstrates that for every $\eps>0$, there is a shh-immune set (in fact, a $\Delta^0_2$ hh-immune set) with upper density at least $1-\eps$. In particular, shh-immunity does not imply intrinsic density 0, even for $\Delta^0_2$ sets.
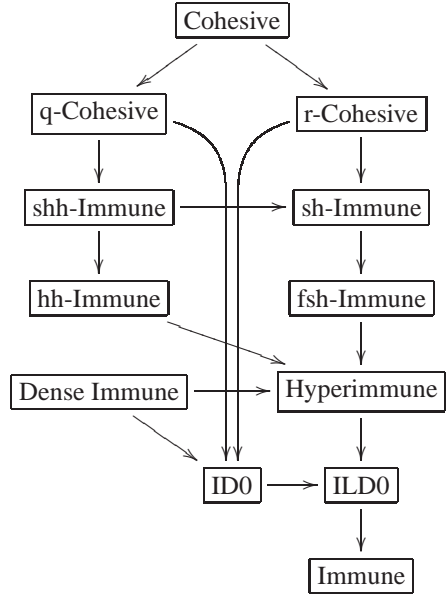
\begin{figure} 
\centering
\resizebox{0.37\linewidth}{!}{
\xymatrix@R-8pt@C=-10pt{
	& *+[F]+\hbox{Cohesive} \ar[dl] \ar[dr] & \\
	*+[F]+\hbox{q-Cohesive} \ar[d] \ar@`{[0,1]-<0.5ex,0pt>,[1,1]-<0.5ex,0pt>,[2,1]-<0.5ex,0pt>}[4,1]!<-0.5ex,0pt> & & *+[F]+\hbox{r-Cohesive} \ar[d] \ar@`{[0,-1]+<0.5ex,0pt>,[1,-1]+<0.5ex,0pt>,[2,-1]+<0.5ex,0pt>}[4,-1]!<0.5ex,0pt>\\
	*+[F]+\hbox{shh-Immune} \ar[d] \ar[rr] & & *+[F]+\hbox{sh-Immune} \ar[d]\\
	*+[F]+\hbox{hh-Immune} \ar[drr] & & *+[F]+\hbox{fsh-Immune} \ar[d]\\
	*+[F]+\hbox{Dense Immune} \ar[dr] \ar[rr] & & *+[F]+\hbox{Hyperimmune} \ar[d]\\
	& *+[F]+\hbox{ID0} \ar[r] & *+[F]+\hbox{ILD0} \ar[d]\\
	& & *+[F]+\hbox{Immune}
	}
}
\caption{The graph of implications between the classical immunity properties and intrinsic density 0. Again, for $\Delta^0_2$ sets, shh-immunity and hh-immunity become equivalent; all other implications are as depicted. (We abbreviate intrinsic [lower] density 0 for infinite sets by I[L]D0.)}
\label{fig:densityImmunityImplications}
\end{figure}

Combining these counterexamples with our results above, we exhaust all possible implications between intrinsic density 0, intrinsic lower density 0, and the standard immunity properties. The graph of the resulting implications for infinite sets is shown in Figure~\ref{fig:densityImmunityImplications}; all implications depicted are strict, and counterexamples are discussed above for all arrows not present in the diagram.

Unfortunately, in the co-\ce{} case (well-studied due to Post's Program), the majority of our proofs of failures of implication collapse. Since hh-immunity does imply dense immunity for co-\ce{} sets, it seems unlikely that our proof method from Theorem~\ref{thm:hhImmune} will help separate the higher immunity properties from intrinsic density 0. In fact, most of our other failures of implication are exhibited by 1-generics or derived from 1-randoms, examples that are inherently not co-\ce{}. This leaves the co-\ce{} diagram incomplete, with a few interesting open questions.

\begin{openQuestion}
Is there an infinite \ce{} set with intrinsic density 1 that is not hypersimple? (f)sh-simple?
\end{openQuestion}

\begin{openQuestion}
Is there a hypersimple set with lower density less than 1? Equal to 0? For all $\eps>0$, is there a hypersimple set with lower density less than $\eps$?
\end{openQuestion}

\section{Intrinsic density and randomness}
\label{sec:idRandomness}

Let us move from the extremes of density (density~0 or 1) to the intermediate densities, as exemplified by density~$\frac{1}{2}$.

The notion of ``density~$\frac{1}{2}$'' is easily recognized as the Law of Large Numbers, as applied to a sequence of flips of a fair coin. We might hope that having density~$\frac{1}{2}$ would be in some way related to a randomness-theoretic property, and stochasticity is the obvious candidate. This follows von~Mises \cite{vonMises} in establishing the existence of limiting frequencies as the key property of a random sequence and, more specifically, the preservation of limiting frequencies under place-selection rules that determine the next bit sampled based only on the values previously sampled. If $\mathcal{C}$ is such a class of selection rules, we say that a sequence $S$ is $\mathcal{C}$-stochastic if no selection rule in $\mathcal{C}$ can select a biased (non-density-$\frac{1}{2}$) subsequence from $S$. We say that a selection rule is monotonic if the places it selects are always in increasing order, and oblivious if the places it selects are independent of $S$, the sequence subject to the selection rule.

There are several standard notions of stochasticity that will be useful to keep in mind. Church-stochastic sequences are stochastic under computable monotonic selection rules, whereas von~Mises-Wald-Church-stochastic sequences are stochastic under partial computable monotonic selection rules. By this definition, sets with density $\frac{1}{2}$ might be termed ``trivially stochastic''; that is to say, they are unbiased under the single selection rule that selects all positions in order. However, this is rarely considered, as stochasticity is generally taken to require selection of proper subsequences.

Passing to intrinsic density $\frac{1}{2}$, we find something more practical: stochasticity under the class of all computable permutations, represented as oblivious selection rules. In fact, this is the class of non-monotonic oblivious selection rules that must eventually select every position. The corresponding notion of randomness, that no computable martingale succeeds on the sequence of bits selected by such a rule, is permutation randomness as defined by Miller and Nies \cite{permutationRandomness}; intrinsic density~$\frac{1}{2}$ is thus the natural notion of permutation stochasticity.

As mentioned above, stochasticity is generally taken to require selection of proper subsequences to preserve density~$\frac{1}{2}$; this would seem to be an obstacle to considering intrinsic density~$\frac{1}{2}$ as a valid notion of stochasticity. Fortunately, permutation stochasticity in fact ensures that many proper subsequences are also unbiased, including all computably-sampled subsequences. We can be fully precise about this with one more definition in hand, and a combinatorial lemma to relate it to our previous work.

\begin{defn}
Given a total computable injection $p$ and an infinite binary sequence $X(n)$, we say that the \emph{subsequence of $X$ sampled by $p$} is
\begin{equation*}
p^{-1}(S)=\set{X(p(n))}_{n\in\omega}.
\end{equation*}
Abusing notation as noted in the introduction, we can apply this directly to any set $S\subseteq\omega$. In set notation, this gives
\begin{equation*}
p^{-1}(S)=\setbuildc{n\in\omega}{p(n)\in S}.
\end{equation*}
\end{defn}

However, even though this new method of sampling generalizes our previous method of considering sets under computable permutations of $\omega$, it has no additional power as far as density is concerned.

\begin{lemma}\label{lem:injectionDensity}
Given any total computable injection $p$, there is a computable permutation $\pi$ such that, for any set $S$, $\pi^{-1}(S)$ has upper and lower density equal to those of $p^{-1}(S)$.
\end{lemma}
\begin{proof}
Given a total computable injection $p$, we define a computable permutation $\pi$ by assigning values $\pi(j)$ in increasing order of $j$. If $j$ is a non-square integer, and $p(j)$ has not already been assigned to $\pi(i)$ for some $i<j$, define $\pi(j)=p(j)$. Otherwise, define $\pi(j)$ to be the least value not assigned to any $\pi(i)$ with $i<j$.

The sizes of $\pi(\left[0,n\right))\cap S$ and $p(\left[0,n\right))\cap S$ differ by at most $\ceil{\sqrt{n}}$. Thus, $\rho_n(\pi^{-1}(S))$ differs from $\rho_n(p^{-1}(S))$ by less than $\frac{2}{\sqrt{n}}$. Therefore,
\begin{equation*}
\limsup_{n\to\infty}{\rho_n(\pi^{-1}(S))}=\limsup_{n\to\infty}{\rho_n(p^{-1}(S))},
\end{equation*}
and
\begin{equation*}
\liminf_{n\to\infty}{\rho_n(\pi^{-1}(S))}=\liminf_{n\to\infty}{\rho_n(p^{-1}(S))}.
\end{equation*}
\end{proof}

From this minor lemma, we note that in fact, any set with intrinsic density has constant density not only under all computable permutations of $\omega$, but also under all computable ``samplings'' of $\omega$. To be more precise:

\begin{cor}
A set $A$ has intrinsic density $d$ iff $\rho(p^{-1}(A))=d$ for every total computable injection $p$.
\end{cor}
\begin{proof}
The reverse direction is obvious by definition, since computable permutations of $\omega$ are also total computable injections.

The forward direction is, at this point, also quite straightforward. Fix a total computable injection $p$. By Lemma~\ref{lem:injectionDensity}, there is a computable permutation $\pi$ such that $\pi^{-1}(S)$ has the same upper and lower densities as $p^{-1}(S)$ for any set $S$, and in particular for $A$. Since $A$ has intrinsic density $d$, we know that $\rho(\pi^{-1}(A))=d$, and so that $\rho(p^{-1}(A))=d$.
\end{proof}

This corollary reveals that intrinsic density~$\frac{1}{2}$ coincides with another form of stochasticity: stochasticity under all computable injections, or equivalently the class of all oblivious non-monotonic selection rules. The corresponding notion of randomness is injection randomness, also as defined by Miller and Nies \cite{permutationRandomness}. Thus, we see that:

\begin{cor}
Permutation stochasticity and injection stochasticity coincide, and are both equivalent to intrinsic density~$\frac{1}{2}$.
\end{cor}

Considering this interpretation of intermediate intrinsic densities (strictly between 0 and 1) as a form of stochasticity, we find that intrinsic density provides an interesting link between the immunity properties and randomness-theoretic ideas. As discussed above, intrinsic density 0 is an immunity-type property, and so intrinsic density 1 is a form of co-immunity (or, as it is called for \ce{} sets, simplicity). Thus, intrinsic density illustrates the relations between immunity, randomness, and simplicity, and provides a continuum of intermediate concepts, all of which follow in the spirit of stochasticity as established by von~Mises. This calls our attention to the fact that all of these properties are, in essence, descriptions of unpredictability: a set is immune if it is sufficiently difficult for a computable enumeration to stay within the set, co-immune if it is difficult to avoid the set, and stochastic if it is difficult to achieve any sort of persistent pattern of biased intersection with the set or its complement.

Of course, all of this relies fundamentally on our use of \emph{intrinsic} density. Considering asymptotic density alone, we find no useful connection to computability or randomness. A set with density~0 need not be immune in any useful sense, as is made clear by considering the computable set of perfect squares. Taking the complement, we obtain a set with density~1 that is trivial to avoid. Moreover, density~$\frac{1}{2}$ is a poor notion of randomness, as recognized by and before von~Mises, carrying no real implications for unpredictability; for instance, the set of even numbers is ``stochastic'' in this sense, and yet is trivially predictable.

\section{Strong variants of generic-case computability}
\label{sec:intrinsicComputability}

Having begun by investigating the implications of adding computable invariance to asymptotic density, we end by returning to the motivating problem with which we began: strengthening Jockusch and Schupp's generic-case computability to obtain similar invariance. They defined generic-case computability as follows:

\begin{defn}
A partial function $f:\omega\to\set{0,1}$ is a \emph{partial description} of $A\subset\omega$ if $f(n)=A(n)$ whenever $f(n)$ converges.

We say that $A\subseteq\omega$ is \emph{computable in the generic case}, or \emph{generic-case computable}, if $A$ has a computable partial description with density-1 domain. We call such a description a \emph{generic-case description}. \cite{JSgc}
\end{defn}

In practical terms, the weakness of generic-case computability was shown by Hamkins and Miasnikov \cite{genericHalting}, who demonstrated that, in several reasonable codings, the halting problem is in fact decidable on a set of asymptotic density 1, due to the density of trivially non-halting programs. This suggests that we should strengthen generic-case computability, to avoid rendering the halting problem ``decidable'' for trivial reasons.

Rybalov \cite{stronglyGenericHalting} has shown that if we insist on convergence on a set with density exponentially approaching 1 (also known as strong generic-case computability), then the halting problem is instead undecidable. Of course, his analysis makes use of asymptotic density on the set of Turing programs, considering the programs with at most $n$ non-final states; strong generic-case computability is not directly applicable to arbitrary subsets of $\omega$, so we must look for an alternative approach.

Furthermore, Corollary~\ref{cor:ceDensity} has a somewhat unfortunate consequence for generic-case computability. For any problem, if there is an algorithm that converges on an infinite set of inputs, that algorithm becomes a generic-case solution for the problem under some alternate coding of the input. After all, the domain of the algorithm is necessarily \ce{}; there is therefore some coding of the underlying problem (corresponding to a permutation of $\omega$) under which the algorithm converges on a set of density 1. In other words, most natural problems have generic-case computable solutions (as defined by Jockusch and Schupp \cite{JSgc}) under some computable permutation. This gives us another reason to use a stricter notion of generic-case computability.

Returning to the original definition of generic-case complexity for group-theoretic problems, from Kapovich, Myasnikov, Schupp, and Shpilrain \cite{genericComplexity}, we note that the authors defined a problem in a finitely generated group to have generic-case complexity $\mathcal{C}$ if and only if this complexity is independent of the choice of generating set. They specifically state that, though the worst-case complexity for most group-theoretic problems does not depend on one's choice of generating set, there is no reason to assume that this should also hold for generic-case complexity. As this choice directly corresponds to a coding of the input to the generic-case algorithm, a natural translation would require that our set be generic-case computable under every computable permutation of $\omega$. Equivalently, by the Myhill Isomorphism Theorem, $A$ should not be considered generic-case decidable unless all of the 1-equivalent sets are as well. Fortunately, this coincides with the standard idea that most computability-theoretic definitions are (or ``should be'') invariant under computable permutation.

We will call this new notion intrinsic generic-case computability, as it must be preserved under computable permutations of $\omega$. Below, we propose four definitions, varying in degree of uniformity.

Our weakest candidate notion of intrinsic generic-case computability is the direct translation of the definition by Kapovich, Myasnikov, Schupp, and Shpilrain:
\begin{defn}
A set $A$ is \emph{(weakly) intrinsically generic-case computable} if{}f $\pi(A)$ is generic-case computable for every computable permutation $\pi\!:\omega\to\omega$.
\end{defn}
Note that we place no requirements on the relationships between the generic-case descriptions for each such image $\pi(A)$; the algorithms may be essentially unrelated.

Insisting on a bare minimum of uniformity, we obtain our next candidate definition:
\begin{defn}
A set $A$ is \emph{(uniformly) intrinsically generic-case computable} if{}f there is a uniformly computable family of functions $f_e$ such that, if $\varphi_e$ is a computable permutation, $f_e$ is a generic-case description of $\varphi_e(A)$; that is, $f_e$ has density-1 domain and wherever $f_e(n)$ converges, it converges to $\paren*{\varphi_e(A)}(n)$.
\end{defn}

On the other hand, allowing our description to require an index may weaken our notion of uniformity; after all, this means that our description $f$ cannot be given only a black-box oracle specifying the computable permutation, but actually requires knowledge of \emph{how} the permutation can be computed --- and in particular may depend on the specific program provided to compute the permutation.

Requiring the description to work with only an oracle might seem a trivial variation, but significant differences have been observed in analogous situations; specifically, in computable model theory, the index-based definition of uniform computable categoricity has been shown to be strictly weaker (and less natural) than the definition providing only an oracle. \cite{ucc} (In general, any oracle-based definition must be at least as strong as the corresponding index-based definition, since it is well-established that there is a Turing-machine procedure allowing us to convert an index into an effective oracle.) We therefore include this option in our list of candidate notions. In this case, we would say that:
\begin{defn}
A set $A$ is \emph{(oracle) intrinsically generic-case computable} if{}f there is a Turing functional $\Phi^X$ such that, for any computable permutation $\pi$ (represented as a set of pairs), $\Phi^{\pi}$ is a generic-case description of $\pi(A)$.
\end{defn}

Finally, we might insist on complete uniformity, and require that a single algorithm provide a description of $A$ on a set that has density 1 under all computable permutations; in other words, that the algorithm converge on a set of \emph{intrinsic} density 1.
\begin{defn}
A set $A$ is \emph{(strongly) intrinsically generic-case computable} if{}f it has a description $\varphi_e$ that converges on a set of intrinsic density 1. (Equivalently, $\varphi_e\circ\pi^{-1}$ is a generic-case description of $\pi(A)$ for all computable permutations $\pi$.)
\end{defn}
Since r-maximal sets are \ce{} and have intrinsic density 1, any r-maximal set is in fact strongly intrinsically generic-case computable. This provides a convenient demonstration that even this strongest definition is weaker than ordinary computability.

More work will be required to distinguish these definitions of intrinsic generic-case computability, and some of them may prove to be equivalent. At this point, though, there are no reasons to presume any equivalences. The author personally expects that the uniform and strong definitions of intrinsic generic-case computability will be the most useful of these four.

On the other hand, even our weakest definition has a certain demonstrable strength. A set $S\subseteq\omega$ is said to be an \emph{index set} if $S(e)=S(e')$ for all $e,e'\in\omega$ where $e$ and $e'$ are indices for equivalent Turing machines. Rice's Theorem \cite{riceTheorem} states that the only computable index sets are $\emptyset$ and $\omega$. We can easily extend this to intrinsic generic-case computability, showing that no non-trivial index set can be weakly intrinsically generic-case computable. Therefore, the halting problem is not intrinsically generic-case computable under any of these definitions.

\begin{thm}\label{thm:indexSetsNotIGC}
Suppose $S\subseteq\omega$ is an index set (\ie{}, $S(e)=S(e')$ for all $e,e'$ such that $\varphi_e=\varphi_{e'}$). $\pi(S)$ is generic-case computable for all computable permutations $\pi$ if{}f $S$ is computable, and thus if{}f $S=\emptyset$ or $S=\omega$.
\end{thm}
\begin{proof}[Sketch of Proof]
The reverse implication is obvious; we will only consider the forward implication.

By the Padding Lemma for Turing machines \cite{soareREsets}, for any $e$, we can enumerate a set $\set{x_{e,0}=e,x_{e,1},\ldots}$ such that $\varphi_e=\varphi_{x_{e,i}}$ for all $i$. Consider the computable sets $R_e=\setbuild*{2^em}{m\text{ odd}}$, and note that $\rho(R_e)=2^{-e-1}$; these, along with the singleton $\set{0}$, comprise a partition of $\omega$. One can easily construct a computable permutation of $\omega$ such that $\pi^{-1}(R_{e+1})\subseteq\set{x_{e,1},x_{e,2},\ldots}$ for all $e$, and $R_0$ is filled with the ``waste'' of the process.

Suppose $\pi(S)$ is generic-case computable, with generic-case description $\Psi$. Since $R_{e+1}$ has positive density for all $e$, there must be some $k_e\in R_{e+1}$ for which $\Psi(k_e)$ converges. To determine whether $e\in S$, wait until $\Psi(k_e)$ converges for some such $k_e$; we then have $S(e)=S(\pi^{-1}(k_e))=\Psi(k_e)$, since $\varphi_{\pi^{-1}(k_e)}=\varphi_e$. This shows that $S$ is computable, so by Rice's Theorem, $S=\emptyset$ or $S=\omega$.
\end{proof}

\begin{cor}
The halting problem is not (weakly) intrinsically generic-case computable.
\end{cor}
\begin{proof}
The halting problem is 1-equivalent to a non-computable index set (\eg{}, $\setbuild*{e}{\paren*{\exists x}\bracket*{\varphi_e(x)\!\downarrow}}$). By the Myhill Isomorphism Theorem \cite{myhillCreative}, this means that its image under some computable permutation is a non-computable index set. Composing this with the permutation from the proof of Theorem~\ref{thm:indexSetsNotIGC}, we obtain a computable permutation under which the image of the halting problem is not generic-case computable. Thus, the halting problem is not even weakly intrinsically generic-case computable.
\end{proof}

\section*{Acknowledgements}

Significant thanks are owed to Denis Hirschfeldt and Robert Soare for innumerable helpful discussions on both the broad strokes and details of the work found herein; in particular, the proof of Theorem~\ref{thm:hhImmune} is based on an approach suggested by Hirschfeldt. The author is also grateful to Carl Jockusch, Paul Schupp, and Rod Downey, for the work that inspired this research and for several conversations on the specifics of this investigation, and to an anonymous referee, whose suggestions have added substantial clarity to this paper.

\bibliographystyle{amsplain}

\end{document}